\documentclass[a4paper,11pt,reqno]{amsart}

\usepackage{graphicx}
\usepackage{a4wide}

\usepackage{eucal}
\usepackage[pdftex,colorlinks]{hyperref}

\usepackage{amsmath}
\usepackage{amssymb}
\usepackage{xspace,epsfig}
\usepackage{mathptmx}

\numberwithin{equation}{section}

\newcommand{\R}{\mathbb R}
\newcommand{\C}{\mathbb C}
\newcommand{\N}{\mathbb N}

\newcommand{\be}{\begin{equation}}
\newcommand{\ee}{\end{equation}}
\newcommand{\ba}{\begin{eqnarray}}
\newcommand{\ea}{\end{eqnarray}}

\usepackage{hyperref}

\newtheorem{thm}{Theorem}[section]
\newtheorem{lem}{Lemma}[section]
\newtheorem{prop}{Proposition}[section]
\newtheorem{rem}{Remark}[section]
\newtheorem{cor}{Corollary}[section]

\numberwithin{equation}{section}

\begin{document}
\title{Exact controllability of a linear Korteweg-de Vries equation by the flatness approach}

\author[Martin]{Philippe Martin}
\address{Centre Automatique et Syst\`emes (CAS), MINES ParisTech, PSL Research University, 60 Boulevard Saint-Michel, 75272 Paris Cedex 06, France}
\email{philippe.martin@mines-paristech.fr}

\author[Rivas]{Ivonne Rivas}
\address{Universidad del Valle, Ciudadela Universitaria Mel\'endez, Calle 13 No. 100--00, A.A. 25360, Cali, Colombia}
 \email{ivonne.rivasl@correounivalle.edu.co}

\author[Rosier]{Lionel Rosier}
\address{Centre Automatique et Syst\`emes (CAS) and Centre de Robotique, MINES ParisTech, PSL Research University, 60 Boulevard Saint-Michel, 75272 Paris Cedex 06, France}
\email{lionel.rosier@mines-paristech.fr}

\author[Rouchon]{Pierre Rouchon}
\address{Centre Automatique et Syst\`emes (CAS), MINES ParisTech, PSL Research University, 60 Boulevard Saint-Michel, 75272 Paris Cedex 06, France}
\email{pierre.rouchon@mines-paristech.fr}

\date{}

\maketitle

\begin{abstract}
We consider a linear Korteweg-de Vries equation on a bounded domain with a left Dirichlet boundary control.
The controllability to the trajectories of such a system was proved in the last decade by using Carleman estimates.
Here, we go a step further by establishing the exact controllability in a space of analytic functions with the aid of the flatness approach.   
\end{abstract}

\vspace{0.3cm}


\vspace{0.5cm}


%%%%%%%%%%%%%%%%%%%%%%%%%%%%%%%%%%%%%%%%%%%%%%%%%%%%%%%%%%%%%%%%%%%%%%%%%%%%%%%%%%%%%%%%%%%%%%%%%%%%%%%%%%%%%%%%%%%%%%%%%%%%%%%

\section{Introduction}
The Korteweg-de Vries (KdV) equation is a well-known dispersive equation that may serve as a model for the propagation of 
gravity waves on the surface of a canal or a lake.  It reads
\be
\label{I1}
\partial _ty + \partial _x^3 y + y\partial _x y + \partial _x y =0, 
\ee 
where $t$ is time, $x$ is the horizontal spatial coordinate,  and $y=y(x,t)$ stands for the deviation of the fluid surface from rest position. 
As usual, $\partial _t y=\partial y/\partial t$, $\partial _x y=\partial y/\partial x$, $\partial _x^3y=\partial ^3 y /\partial x^3$,  etc.  

When the equation is considered on a bounded interval $(0,L)$, it has to be supplemented with three boundary conditions, for instance
\be
\label{I2} 
y(0,t)=u(t),\quad y(L,t)=v(t),\quad \partial _x y(L,t)=w(t), 
\ee 
and an initial condition
\be
\label{I3}
y(x,0)=y_0(x). 
\ee
 
The controllability of the Korteweg-de Vries equation with various boundary controls has been considered by many authors 
since several decades (see e.g. the surveys \cite{RZsurvey,cerpa}).  The exact controllability in the energy space $L^2(0,L)$ 
 was derived by Rosier in \cite{R1997}
(resp. by Glass and Guerrero in \cite{GG2010})  with $w$ as the only control input (resp. with $v$ as the only control input). On the other hand, 
if we take $u$ as the only control input, the exact controllability fails in the energy space  \cite{R2004}, because of the smoothing effect. 
Nevertheless, both  the null-controllability and the controllability to the trajectories hold with the left Dirichlet boundary control, see \cite{R2004} and \cite{GG2008}.  
The aim of the present paper is to go a step further by investigating the exact controllability in a ``narrow'' space with the left Dirichlet boundary control. 
Due to the smoothing effect, the space in which the exact controllability  can hold is a space of {\em analytic functions}. 
For the sake of simplicity, we will focus on a linear KdV equation (removing the nonlinear term $y\partial _x y$). Performing a scaling 
in time and space, there is no loss of generality in assuming that $L=1$. By a translation, we can also assume that $x\in (-1,0)$. 
The first-order derivative term will assume the form 
$a\partial _x y$ where $a\in \R _+$ is some constant. The case $a=1$ corresponds to the linearized KdV equation 
\be
\label{I4}
\partial _ty + \partial _x^3 y +   \partial _x y =0, 
\ee
while the case $a=0$ corresponds to the ``simplified'' linearized KdV equation 
\be
\label{I5}
\partial _ty + \partial _x^3 y  =0,
\ee
which is often considered when investigating the Cauchy problem on the line  $\R$ (instead of a bounded interval) by doing the change of unknown $\tilde y(x,t)=y(x+t,t)$. 

The paper will be concerned with the control properties  of the system: 
\ba
\partial _t y + \partial _x ^3 y + a\partial _x y =0,&& x\in (-1,0),\ t\in (0,T), \label{A1}\\
y(0,t)=\partial _x y(0,t)=0,&&  t\in (0,T), \label{A2}\\
y(-1,t)=u(t),&& t\in (0,T), \label{A3}\\
y(x,0)=y_0(x),&& x\in (-1,0), \label{A4}
\ea
where $y_0=y_0(x) $ is the initial data and $u=u(t)$ is the control input. 

We shall address the following issues: \\
1. (Null controllability) Given any $y_0\in L^2(-1,0)$, can we find a control $u$ such that the solution $y$ of \eqref{A1}-\eqref{A4} 
satisfies $y(.,T)=0$?\\
2. (Reachable states) Given any $y_1\in {\mathcal R}$ (a subspace of $L^2(-1,0)$ defined thereafter), can we find a control $u$ such that the solution $y$ of \eqref{A1}-\eqref{A4} with $y_0=0$ satisfies $y(.,T)=y_1$?\\    
We shall investigate both issues by the flatness approach and derive an exact controllability in $\mathcal R$ by combining
our results. 

The null controllability of \eqref{A1}-\eqref{A4} was established in \cite{R2004} (see also \cite{GG2008}) by using a Carleman estimate. 
The control input $u$ was found in a Sobolev space (e.g. $u\in H^{ \frac{1}{2} -\varepsilon}(0,T)$ for all 
$\varepsilon >0$ if 
$y_0\in L^2(-1,0)$, see \cite{GG2008}). Here, we shall improve this result by designing a control input in a Gevrey class. 
Furthermore, the trajectory and the control will be given explicitly as the sums of series parameterized by the flat output.  
To state our result, we need introduce a few notations. A function $u\in C^\infty ( [ t_1,t_2])$ is said to be {\em Gevrey of 
order $s\ge 0$ on $[t_1,t_2]$}  if there exist some constant $C,R\ge 0$ such that 
\[
\vert \partial _t ^n u(t) \vert \le C \frac{ (n!) ^s}{R^n}\quad \forall n\in \N, \ \forall t\in [t_1,t_2]. 
\] 
The set of functions Gevrey of order $s$ on $[t_1,t_2]$ is denoted by $G^s ([t_1,t_2])$.  A function 
$y\in C^\infty ([x_1,x_2] \times [t_1,t_2] )$ is said to be Gevrey of order $s_1$ in $x$ and $s_2$ in $t$ on $[x_1,x_2]\times[t_1,t_2]$
if there exist some constants $C,R_1,R_2>0$ such that 
\[
\vert \partial _x^{n_1} \partial _t ^{n_2}  y(x,t) \vert \le C \frac{ (n_1 ! )^{s_1}  (n_2 !)^{s_2}  }{R_1^{n_1} R_2^{n_2}}
\quad \forall n_1,n_2\in \N, \ \forall (x,t)\in [x_1,x_2]\times [t_1,t_2]. 
\] 
The set of functions Gevrey of order $s_1$ in $x$ and $s_2$ in $t$  on $[x_1,x_2] \times [t_1,t_2]$ is denoted by 
$G^{s_1,s_2}  ( [x_1,x_2]\times [t_1,t_2])$.  

The first main result in this paper is a null controllability result with a control input in a Gevrey class. 
\begin{thm}
\label{thm1} Let $y_0\in L^2(-1,0)$, $T>0$, and $s\in [\frac{3}{2},3)$. Then there exists a control input $u\in G^s([0,T])$ such that the solution 
$y$ of \eqref{A1}-\eqref{A4} satisfies $y(.,T)=0$. Furthermore,  it holds that
\be
y\in C([0,T], L^2(-1,0))\cap G^{\frac{s}{3}, s} ([-1,0]\times [\varepsilon ,T])\quad \forall \varepsilon \in (0,T). 
\ee 
\end{thm}
The second issue investigated in this paper is the problem of the reachable states. For the heat equation, an important step in the  characterization of the reachable states was given in \cite{MRRreachable} with the aid of the flatness approach.  It was proven there 
that reachable states can be extended as holomorphic functions on some square of the complex plane, and conversely 
that holomorphic functions defined on a ball centered at the origin and with a sufficiently large radius give by restriction to the real line reachable states. See also \cite{DE} for an improvement of this result as far as the domain of analyticity of the reachable states is concerned.  

To the best knowledge of the authors, the determination of the reachable states for \eqref{A1}-\eqref{A4}  has not been addressed so far. 
From the controllability to the trajectories established in \cite{R2004,GG2008}, we know only that any function $y_1=y_1(x)$
that can be written as $y_1(x)=\bar y(x,T)$ for some trajectory $\bar y$ of \eqref{A1}-\eqref{A4} associated with some  
$y_0\in L^2(-1,0)$ and $u=0$ is reachable. But such a function is in $G^\frac{1}{2} ([-1,0])$\footnote{It is conjectured that it belongs to $G^\frac{1}{3} ([-1,0])$.} with $y_1(-1)=0$ and $(\partial _x ^3 + a\partial _x )^n \, y_1(-1)=0$ for all $n\ge 1$,  according to Proposition \ref{prop1} (see below). Proceeding as in \cite{MRRreachable}, we shall obtain a class of reachable states that are {\em less regular} than those for the controllability to the trajectories (namely $y_1\in G^1 ([-1,0])$ for which no boundary condition has to be imposed at $x=-1$. 

To state our second main result, we need to introduce again some notations. For $z_0\in \C$ and $R>0$, we denote by $D(z_0,R)$ the open disk 
\[
D(z_0,R) :=\{ z\in \C; \ |z-z_0|<R \} ,
\]  
and by $H ( D (z_0,R) )$ the set of holomorphic (i.e. complex analytic) functions on $D(z_0,R)$. Introduce the operator
\[
Py:= \partial _x ^3 y + a \partial _x y,
\]
so that \eqref{A1} can be written
\be
\label{A9}
\partial _t y  + Py=0.
\ee
Since $\partial _t $ and $P$ commute, it follows from \eqref{A9} that for all $n\in \N ^*$
\be
\label{A10}
\partial _t^n  y + (-1) ^{n-1}P^n\,  y=0
\ee
where $P^n=P\circ P^{n-1}$ and $P^0=Id$.  We are in a position to define the set of reachable states: for any $R>1$, let 
\begin{multline*}
{\mathcal R} _R := \{ y\in C^0([-1,0]); \ \exists z\in H( D(0,R) ), \ y=z_{ \vert [-1,0] }, \textrm{ and } 
(P^n\, y) (0)=\partial _x (P^n\, y) (0)=0\ \ \forall n\ge 0\} . 
\end{multline*}

The following result is the second main result in this paper. 
\begin{thm}
\label{thm2}
Let $a\in \R_+$, $T>0$, and $R>R_0:=e ^{(3e)^{-1}} (1+a)^\frac{1}{3} >1$.  Pick any $y_1\in {\mathcal R}_R$. Then there exists 
a control input $u\in G^3([0,T])$ such that the solution $y$ of \eqref{A1}-\eqref{A4} with $y_0=0$ satisfies $y(.,T)=y_1$. Furthermore,
$y\in G^{1,3}([-1,0]\times [0,T])$. 
\end{thm}
\begin{rem}
\begin{enumerate}
\item As for the heat equation, it is likely that any reachable state for the linear Korteweg-de Vries equation
can be extended as an holomorphic function on some open set in $\C$.  
\item The reachable states corresponding to the controllability to the trajectories  are in $G^\frac{1}{2} ( [-1,0] )$, so that they 
can be extended as functions in $H(\C )$. By contrast, the reachable functions in Theorem \ref{thm2} need not be holomorphic on 
the whole set $\C$: they can have poles outside $D(0,R)$.
\item The set ${\mathcal R}_R$ takes a very simple form when $a=0$. Indeed, in that case
\begin{eqnarray*}
{\mathcal R}_R &=& \{ y\in C([-1,0]); \ \exists z\in H(D(0,R)),\ y=z_{ \vert [-1,0]} \textrm{ and } 
 \partial _x ^{3n} y(0)=\partial _x ^{3n+1} y (0)=0\ \  \ \forall n\in \N \}, \\
 &=&  \{ y\in C([-1,0]); \ \exists (a_n)_{n\in \N} \in {\R ^\N}, \ \sum_{n=0}^\infty  | a_n | r^{3n} <\infty \ \forall r\in (0,R) \textrm{ and } \\
 &&\qquad  \qquad\qquad\quad 
 y(x)=\sum_{n=0}^\infty a_n x^{3n+2} \ \ \forall  x\in [-1,0] \} .
\end{eqnarray*} 
Note that $y(-1)$ needs not be  $0$ for $y\in {\mathcal R}_R$. Examples of functions in ${\mathcal R}_R$ include
(i) the polynomial functions of the form $y(x)=\sum_{n=0}^N a_n x^{3n+2}$; (ii) the entire function 
\[
y(x)= e^x + j e^{jx} + j^2 e^{j^2x} 
\] 
where $j:=e^{i\frac{2\pi}{3}}$. Note that $y$ is real-valued and $y(-1)>0$  (see Fig. 1).
\begin{center}
\includegraphics[width=0.3\columnwidth]{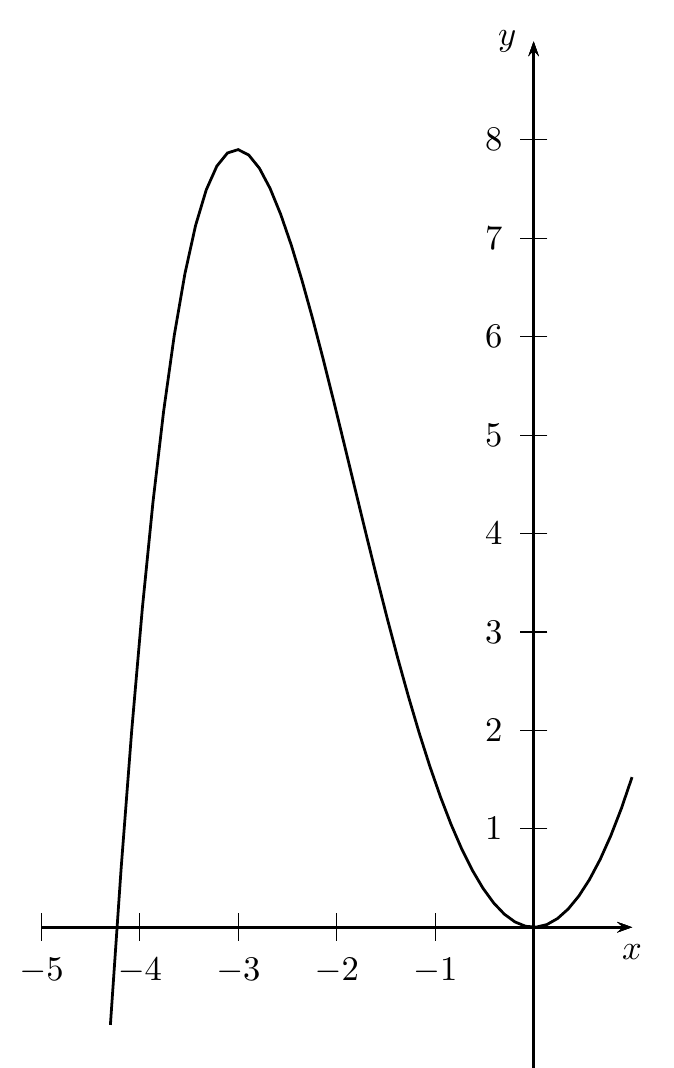}\\
Fig. 1. The function $y(x)=  e^x + j e^{jx} + j^2 e^{j^2x}$. 
 \end{center}

\end{enumerate}
\end{rem}

Combining Theorem \ref{thm1} and Theorem \ref{thm2}, we obtain the 
following result which implies the exact controllability of \eqref{A1}-\eqref{A4} 
in ${\mathcal R}_R$  for $R>R_0$.

\begin{cor}
\label{cor1}
Let $a\in \R_+$, $T>0$, $R>R_0$, $y_0\in L^2(-1,0)$ and $y_1\in {\mathcal R}_R$. 
Then there exists $u\in G^3([0,T])$ such that the solution of \eqref{A1}-\eqref{A4} satisfies  $y(.,T)=y_1$. 
\end{cor}   

Since system \eqref{A1}-\eqref{A4} is linear, it is sufficient to pick $u=u_1+u_2$
where $u_1$ is the control given by Theorem \ref{thm1} for $y_0$ and $u_2$ is the control given by Theorem \ref{thm2} for $y_1$.  

The paper is outlined as follows. Section \ref{section2} is devoted to the null controllability of the linear KdV equation. The flatness property is established in Proposition \ref{prop1}. The smoothing effect for the linear KdV equation 
from $L^2(-1,0)$ to $G^\frac{1}{2}( [-1,0] )$ is derived in Proposition \ref{prop2}. The section ends with the proof of Theorem \ref{thm1}.    
Section \ref{section3} is concerned with the study of the reachable states. The flatness property is extended to the limit case $s=3$ in Proposition \ref{prop3}. Theorem 
\ref{thm2} then follows from Proposition \ref{prop3} and some version of Borel theorem borrowed from \cite{MRRreachable}.

\section{Null controllability by the flatness approach}
\label{section2}

In this section, we are concerned with the null controllability of \eqref{A1}-\eqref{A4}. 
\subsection{Flatness property}
Our first task is to establish the flatness property, namely the 
fact that the solution of \eqref{A1}-\eqref{A3} can be parameterized by the ``flat ouput'' $\partial _x ^2 y(0,.)$. More precisely, we consider the ill-posed system 
\ba
\partial _t y + \partial _x ^3 y + a\partial _x y = 0,&&  x\in (-1,0), \ t\in (0,T), \label{B1}\\
y(0,t)=\partial _x y(0,t)=0, && t\in (0,T), \label{B2}\\
\partial _x ^2 y(0,t)=z(t),&&t\in (0,T),\label{B3}
\ea
and we prove that it admits a solution $y\in G^{\frac{s}{3}, s} ([-1,0]\times [0,T])$ whenever $z\in G^s ([0,T])$ and $1\le s <3$. 

The trajectory $y$ and the control input $u$ can be  written as 
\ba
y(x,t)            &=& \sum_{i\ge 0} g_i (x) z^{(i)} (t), \label{B4}\\
u(t) = y(-1,t) &=& \sum_{i\ge 0} g_i (-1) z^{ ( i ) } (t)  \label{B5}  
\ea
where the {\em generating functions} $g_i$, $i\ge 0$, are defined as in \cite{MRRparabolic}. More precisely, the function $g_0$ is defined as the solution of the Cauchy  problem 
\ba
g_0''' (x) + ag_0' (x)&=& 0, \qquad x\in (-1,0), \label{B6}\\
g_0(0)=g_0'(0)&=&0, \label{B7}\\
g_0''(0)&=&1 \label{B8}
\ea
(where $'=d/dx$), while the function $g_i$ for  $i\ge 1$ is defined inductively as the solution of the Cauchy problem 
\ba
g_i''' (x)+ a g_i' (x)&=& - g_{i-1} (x), \qquad x\in (-1,0), \label{B9}\\
g_i(0)=g_i'(0)=g_i''(0) &=&0. \label{B10}
\ea
It is well known that $g_i$ for $i\ge 1$ can be expressed in terms of $g_0$ and $g_{i-1} $ as 
\be
\label{B11}
g_i(x)= - \int_0^x g_0(x-\xi) g_{i-1}(\xi )\, d\xi.
\ee
\begin{rem}
\begin{enumerate}
\item If $a=0$, then it follows from direct integrations of \eqref{B6}-\eqref{B8} and \eqref{B9}-\eqref{B10} that 
\be
g_i(x)= (-1)^i \frac{x^{3i+2}}{(3i+2)!}, \quad x\in [-1,0], \ i\ge 0. 
\label{B12}
\ee
\item If $a>0$, then the solution of \eqref{B6}-\eqref{B8} reads
\be
g_0(x)= \frac{1}{a} (1-\cos ( \sqrt{a} x)). 
\label{B13}
\ee 
%The functions $g_j$ for $j\ge 1$,  obtained inductively from \eqref{B11}, involve both trigonometric functions and polynomial functions. 
\end{enumerate}
\end{rem}

To ensure the convergence of the series in \eqref{B4}, we first have to establish some estimates for $\Vert g_i\Vert _{L^\infty (-1,0) }$.  

\begin{lem}
\label{lem1}
Let $a\in \R_+$. Then for all $i\ge 0$
\be
\label{B14}
|g_i(x)| \le  \frac{|x| ^{3i+2}}{(3i+2)!} \quad \forall x\in [-1,0].
\ee
\end{lem}
\begin{proof}
If $a=0$, then \eqref{B14} is a direct consequence of \eqref{B12}. Assume now that $a>0$ and let us prove \eqref{B14} by induction on $i$. It follows from \eqref{B13} that 
\be
0\le g_0(x) \le \frac{x^2}{2}, \quad \forall x\in  [-1,0],
\label{B15}
\ee
so that \eqref{B14} is true for $i=0$. 
Assume now that \eqref{B14} is true for some $i-1\ge 0$. Then, integrating by parts twice in \eqref{B11} and using \eqref{B7}, we see that 
\begin{eqnarray*}
	g_i(x) &=& -\underbrace{\Bigl[g_0(x-\xi)\int_0^\xi g_{i-1}(\zeta)d\zeta\Bigr]_0^x}_{=0}
	\quad -\,\int_0^xg_0'(x-\xi) \big( \int_0^\xi g_{i-1}(\zeta)d\zeta \big) d\xi\\
	&=& -\underbrace{\Bigl[g_0'(x-\xi)\int_0^\xi \big( \int_0^\zeta g_{i-1}(\sigma )d\sigma  \big) d\zeta\Bigr]_0^x}_{=0}
	\quad - \,\int_0^x\underbrace{g_0''(x-\xi)}_{\cos \big( \sqrt{a} ( x-\xi) \big) }   \big( \int_0^\xi \big( \int_0^\zeta g_{i-1}(\sigma )d\sigma\big)  d\zeta \big) d\xi.
\end{eqnarray*}
It follows that 
\[
|g_i (x)| \le \left\vert \int_0^x   \big( \int_0^\xi \big( \int_0^\zeta  |g_{i-1}(\sigma ) | d\sigma\big)  d\zeta \big) d\xi  \right\vert
\le \left\vert \int_0^x   \big( \int_0^\xi \big( \int_0^\zeta  \frac{\sigma ^{3i-1} }{ (3i-1) ! } d\sigma\big)  d\zeta \big) d\xi  \right\vert 
= \frac{ |x|^{3i+2}} { (3i+2)! },
\]
as desired.
\end{proof}

We are now in a position to solve system \eqref{B1}-\eqref{B3}. 
\begin{prop}
\label{prop1}
Let $s\in [1,3)$, $z\in G^s ([0,T])$, and $y=y(x,t)$ be as in \eqref{B4}. Then $y\in G^{\frac{s}{3}, s} ( [-1,0]\times [0,T] )$ and it solves \eqref{B1}-\eqref{B3}.  
\end{prop}
\begin{proof}
We need to estimate the behavior of the constants in the equivalence of norms in $W^{n,p}([-1,0])$ as $n\to \infty$. 
For $n\in \N$, $p\in [1,\infty ]$, and $f\in W^{n,p}(-1,0)$, we denote $\Vert f\Vert _p = \Vert f\Vert _{L^p(-1,0)}$ and 
\[
\Vert f\Vert _{n,p} = \sum_{i=0}^n \Vert \partial _x ^i \, f\Vert _{p}. 
\]
The following result will be used several times. Its proof is given in appendix, for the sake of completeness. 
\begin{lem}
\label{lem10} Let $p\in [1,\infty]$ and $P=\partial _x ^3 + a\partial _x$, where $a\in \R _+$. Then there exists a constant $K=K(p,a)>0$ such that 
for all $n\in \N$, 
\be
\label{P1}
(1+\frac{1}{a}) ^{-1}(1+a)^{-n} \sum_{i=0}^n \Vert P^i f\Vert _p \le \Vert f\Vert _{3n,p} \le K^n \sum_{i=0}^n \Vert P^i f \Vert _p, \quad \forall f\in W^{3n,p}(-1,0).   
\ee
\end{lem}  
We follow closely \cite{MRRheat}. Pick any $z\in G^s([0,T])$ for some $s\in [0,3)$. We can find some numbers 
$M>0$ and $R<1$ such that 
\[
\vert z^{(i)} (t) \vert \le M\frac{(i!)^s}{R^i}, \quad \forall i\in \N, \ \forall t\in [0,T].
\]
Pick any $m,n\in \N$. Then 
\be
\label{Q1}
\partial _t ^m P^n ( g_i(x) z^{ (i) } (t) ) = 
\left\{ 
\begin{array}{ll}
z^{(i+m)} (t) (-1)^n g_{i-n} (x) &\textrm{ if } i-n\ge 0, \\
0 &\textrm{ if } i-n<0. 
\end{array}
\right.  
\ee
Assume that $i\ge n$. Setting $j=i-n$ and $N=n+m$, so that $j+N=i+m$, we have that 
\[
\left\vert 
\partial _t ^m P^n \big(  g_i(x) z^{ (i) } (t) \big)
\right\vert 
\le M \frac{(i+m)! ^s}{R^{i+m}} \frac{1}{(3(i-n)+2)!} \le M \frac{(j+N)! ^s}{R^{j+N}} \frac{1}{(3j+2)!} \cdot
\]
Let 
\[
S:=\sum_{i\ge n} |\partial _t ^m P^n (g_i(x)z^{(i)} (t)) | .
\]
Using the classical estimate $ (j+N) ! \le 2^{j+N} j ! \, N ! $ and the equivalence $(3j) ! \sim  3^{3j+\frac{1}{2} } (2\pi j)^{-1} (j!)^3$ which follows at once from Stirling formula, we obtain that 
\begin{eqnarray*}
S  &\le& M\, \sum_{j\ge 0} \frac{(j+N)!^s }{R^{j+N}} \frac{1}{(3j+2)!}\\
&\le& M'\,  \sum_{j\ge 0} \frac{ (2^{j+N} j!\, N!)^s }{ R^{j+N}} 
\frac{2\pi (j+1)} { (3j+2)(3j+1) 3^{3j+\frac{1}{2}}  ( j !)^3} \\
&\le& M'' \frac{ (N !)^s}{ (\frac{R}{2^s})^N}
\end{eqnarray*} 
for some positive constants $M',M"$. Indeed, since $s<3$, we have that 
\[
\sum_{j\ge 0} \frac{ (2^j j!\, )^s }{ R^{j}} 
\frac{2\pi (j+1)} { (3j+2)(3j+1) 3^{3j+\frac{1}{2}} ( j !)^3} < +\infty . 
\]
Using again the fact that $N!=(n+m)!\le 2^{n+m}n!\, m!$, we arrive to 
\[
S\le M''  \frac{ n! ^s m ! ^s }{ R_2^n R_1^m}
\]
with $R_1=R_2=R/2^{s}$. 
 Let $K$ be as in Lemma \ref{lem10} for $p=\infty$. Then we have
\begin{eqnarray*}
\sum_{i\ge n} \Vert \partial _t ^m ( g_i(x) z^{( i) } (t) )\Vert _{3n,\infty} 
&\le & K^n \sum_{i\ge n} \ \sum_{0\le j\le n} \Vert \partial _t ^m P^j (g_i(x)z^{(i)} (t))\Vert _\infty \\
&\le&  M'' K^n \sum_{0\le j\le n} \left(  \frac{ j! ^s m ! ^s }{ R_2^j R_1^m} \right) \\
&\le& M'''   \frac{ n!^s }{R_2'^n} \frac{m!^s}{R_1^m} 
\end{eqnarray*}
for some $R_2'>0$ and some $M'''>0$. 
This shows that the series of derivatives $\partial _t ^m\partial _x ^l\big( g_i (x)z^{ (i) } (t) \big)$ is uniformly convergent 
on $[-1,0]\times [0,T]$ for all $m,l\in \N$, so that $y\in C^\infty ([-1,0]\times [0,T])$ and it satisfies for $l\le 3n$
\[
|\partial _t ^m \partial _x ^l \, y(x,t) |\le M'''   \frac{ n!^s }{R_2'^n} \frac{m!^s}{R_1^m} \quad \forall x\in [-1,0], \ \forall t\in [0,T]
\] 
for some constant $M'''>0$. 
Note that 
\[
n! ^s \sim \left( \frac{ 2\pi n (3n)! }{3^{3n+\frac{1}{2}}} \right) ^\frac{s}{3} \cdot
\]
It follows that if $l\in \{ 3n-2,3n-1,3n\}$, then 
$n! ^s / R_2'^n \le C' l! ^\frac{s}{3} /R_2''^l $ for some $C'>0$ and  $R_2''>0$. This yields 
\[
\vert \partial _t^m\partial _x ^l \, y(x,t) \vert \le C'M''' \frac{m! ^s}{R_1^m} \frac{l !^\frac{s}{3} }{R_2''^l}, \quad \forall x\in [-1,0], \  \forall t\in [0,T], 
\]
as desired. The fact that $y$ solves  \eqref{B1}-\eqref{B3} is obvious.
\end{proof}

\subsection{Smoothing effect}
We now turn our attention to the smoothing effect. We show that any solution $y$ of the
initial value problem \eqref{A1}-\eqref{A4} with $u\equiv 0$ and $y_0\in L^2(-1,0)$
is a function  Gevrey of order $1/2$ in $x$ and $3/2$ in $t$ for $t>0$. 
\begin{prop}
\label{prop2}
Let $y_0\in L^2(-1,0)$ and $u(t)=0$ for $t\in \R _+$. Then the solution
$y$ of \eqref{A1}-\eqref{A4} satisfies $y\in G^{\frac{1}{2},\frac{3}{2}}( [-1,0]\times [\varepsilon, T])$ for all $0<\varepsilon <T<\infty$. More precisely, there exist
some positive constant $K,R_1,R_2$ such that 
\be
\label{smoothing}
\vert \partial _t^n\partial _x^p \, y(x,t)\vert \le 
K t^{-\frac{3n+p+3}{2}} \frac{n! ^\frac{3}{2}}{R_1^n}\frac{p!^\frac{1}{2}}{R_2^p}\quad \forall p,n\in \N ,\   \forall t\in (0,T], \ \forall x\in [-1,0]. 
\ee
\label{prop2}
\end{prop}
\begin{proof}
Using \eqref{smoothing} on intervals of length one, we can, without loss of generality, assume that $T=1$.

Let us introduce the operator $Ay= - Py= -\partial _x ^3 y -a\partial _x y$ with domain 
\[
D(A)= \{ y\in H^3(-1,0); \ \ y(-1)=y(0)=\partial _x y (0)=0\} \subset L^2(-1,0).  
\] 
Then it follows from \cite{R1997} that $A$ generates a semigroup of contractions in $L^2(-1,0)$, and tht a global Kato smoothing effect holds. More precisely, if 
$y=e^{tA}y_0$ is the mild solution issuing from $y_0$ at $t=0$, then we have for all $T>0$
\ba
\Vert y(T)\Vert _{L^2} &\le & \Vert y_0\Vert _{L^2} \label{E1}\\
\int_0^T \Vert \partial _x  y(.,t)\Vert ^2_{L^2} \, dt &\le& \frac{1}{3} (aT+1)
\Vert y_0\Vert ^2_{L^2}, \label{E2}
\ea  
where $\Vert f\Vert _{L^2}=(\int_{-1}^0 |f(x)|^2 dx)^\frac{1}{2}$. For simplicity, we denote 
$\Vert f\Vert _{H^p}=(\sum _{i=0}^p \Vert \partial _x ^i f\Vert ^2_{L^2} )^\frac{1}{2}$
for $p\in \N$.

For $p\in \{ 0, 1, 2, 3, 4\}$, we introduce the Banach spaces 
\[
X_0=L^2(-1,0), \ X_1=H^1_0(-1,0), \ X_2=\{ y\in H^2(-1,0); \ y(-1)=y(0)=\partial _x y(0)=0\}, \ X_3=D(A), 
\] 
and 
\[
X_4=\{ y\in H^4(-1,0), \ \ y(-1)=y(0)=\partial _x y(0)=\partial _x ^3 y(-1)=\partial _x ^3 y(0)=0 \}, 
\]
$X_p$ being endowed with the norm $\Vert \cdot \Vert _{H^p}$ for  $p=0, ..., 4$. 

\noindent 
{\sc Claim 1} There is a constant $C=C(a)>0$ such that 
\be
\label{E3}
\Vert y(.,t)\Vert _{H^1} \le \frac{C}{\sqrt{t}} \Vert y_0\Vert _{L^2} \quad \forall t\in (0,1]. 
\ee
To prove Claim 1, we follow closely \cite{PR}. 
Pick any $y_0\in X_3=D(A)$, we have by a classical property from semigroup theory that $y\in C([0,T], D(A))$, and that $z(.,t)=Ay(.,t)$ satisfies $z(.,t)=e^{tA} z(.,0)$. It 
follows by \eqref{E1} that  
\[
\Vert z(.,t)\Vert _{L^2} \le \Vert z(.,0)\Vert _{L^2} \quad \forall t\in (0,1]. 
\] 
Summing with \eqref{E1}, this yields
\[
\Vert y(.,t)\Vert _{L^2} + \Vert P y(.,t) \Vert _{L^2} 
\le \Vert y(.,t)\Vert _{L^2} +  \Vert P y_0 \Vert _{L^2}, \quad t\in (0,1].  
\] 
Using Lemma \ref{lem22}, this yields 
\be
\Vert y_0 \Vert _{H^3} \le C_3 \Vert y_0 \Vert _{H^3}, \quad t\in (0,1]
\ee
for some $C_3=C_3(a)>0$. Using interpolation, and noticing that $x_1=[X_0,X_3]_\frac{1}{3}$, we infer the existence of some constant $C_1=C_1(a)>0$ such that 
\be
\Vert y(.,t)\Vert _{H^1} \le C_1 \Vert y_0\Vert _{H^1} \quad \forall y_0\in X_1, \ \forall t\in (0,1]. 
\ee
This yields for $0<s<t \le 1$ $\Vert y(.,t)\Vert ^2 _{H^1} \le C_1^2  \Vert y(.,s)\Vert ^2_{H^1}$, 
which gives upon integration over $(0,t)$ for $t\in (0,1]$
\[
t\Vert y(.,t)\Vert ^2_{H^1} \le C_1^2 \int_0^t \Vert y(.,s)\Vert ^2_{H^1}ds 
\le \frac{C_1^2}{3} ((a+3)T+1) 
\Vert y_0\Vert ^2_{L^2}
\] 
where we used \eqref{E1}-\eqref{E2}. Thus \eqref{E3} holds. The proof of Claim 1 is achieved.

\noindent
{\sc Claim 2.} 
There is some constant $C>0$ such that 
\be
\label{E6}
\Vert y(.,t) \Vert _{H^{p+1}} \le \frac{C}{\sqrt{t}} \Vert y_0\Vert _{H^p}, 
\quad \textrm{ for } p\in \{ 0,1,2,3\}, \ y\in X_p, \ t\in (0,1].  
\ee
To prove Claim 2, we pick again $y_0\in D(A)$ and set $z(.,t)=Ay (., t)$. We infer from 
\eqref{E3} applied to $z (.,t)$ that 
\[
\Vert A y(.,t)\Vert _{H^1} =\Vert z(.,t)\Vert _{H^1} \le \frac{C}{\sqrt{t}} 
\Vert z(.,0)\Vert _{H^1} =\frac{C}{\sqrt{t}} \Vert A y_0\Vert . 
\] 
Combined with \eqref{E3}, this gives 
\be
\Vert y(.,t)\Vert _{H^1} + \Vert (\partial _x ^3 + a\partial _x )y(.,t)\Vert _{H^1} 
\le \frac{C}{\sqrt{t}} (\Vert y_0\Vert _{L^2} + \Vert A y_0\Vert _{L^2}).  
\label{E40}
\ee
We know from Lemma \ref{lem22} that for $z\in H^3(-1,0)$ 
\[
\Vert z\Vert _{H^3} \le C (\Vert z\Vert _{L^2} + \Vert Pz\Vert _{L^2} ). 
\]
It follows that for $z\in X_4$ ($C$ denoting a positive constant that may
 change from line to line, and that do not depend on $t$ and on $y_0$) 
\begin{eqnarray*}
\Vert z\Vert _{H^4} 
&\le& C (\Vert z\Vert _{H^3} + \Vert \partial _x ^4 z \Vert _{L^2} ) \\
&\le& C \left( \Vert z\Vert _{L^2} + \Vert P z \Vert _{L^2}  
+\Vert \partial _x (\partial _x^3 z + a\partial _x z) \Vert _{L^2} + 
a\Vert \partial _x ^2 z\Vert _{L^2} \right) \\
&\le& C (\Vert z \Vert _{H^1} + \Vert A z\Vert _{H^1} ).   
\end{eqnarray*}
Combined with \eqref{E40}, this gives 
\be
\label{E5}
\Vert y(.,t)\Vert _{H^4} \le \frac{C}{\sqrt{t}} \Vert y_0\Vert _{H^3} \quad \forall t\in (0,1]
\ee
for all $y_0\in X_4$, and also for all $y_0\in X_3=D(A)$ by density. 

Interpolating between \eqref{E3} and \eqref{E5}, we obtain \eqref{E6}. The proof of Claim 2 is achieved. 

Using Claim 2 inductively and spitting $[0,t]$ into $[0,t/3]\cup [t/3 , 2t/3 ]\cup [2t/3,t]$,  we infer that 
\ba
&&\Vert a\partial _x y (.,t)\Vert _{L^2} \le \frac{Ca}{\sqrt{t}} \Vert y_0\Vert _{L^2} \quad t \in (0,1],  \label{E7} \\
&&\Vert \partial _x ^3 y (.,t) \Vert _{L^2} \le \frac{C}{\sqrt{\frac{t}{3}}} \Vert   \partial _x^2 y(., \frac{2t}{3} )\Vert _{L^2}
\le \left( \frac{C}{\sqrt{\frac{t}{3}}} \right)^2  \Vert   \partial _x y(., \frac{t}{3} )\Vert _{L^2}
\le \left( \frac{C}{\sqrt{\frac{t}{3}}} \right)^3  \Vert    y_0\Vert _{L^2} \label{E8}
\ea
Combining \eqref{E7} and \eqref{E8}, we infer the existence of a constant $C'=C'(a)>0$ (say $C'\ge 1$, for simplicity), such that  
\be
\Vert A y(t) \Vert _{L^2} \le \frac{C'}{t^\frac{3}{2}} \Vert y_0\Vert _{L^2}, \quad \textrm{ for } y_0\in L^2(-1,-), \ t\in(0,1].  
\ee
For $y_0\in D(A^{n-1})$, $z(t)=A^{n-1}y(t)$ satisfies $z(.,t)=e^{tA} (A^{n-1} y_0)$ and thus 
\[
\Vert A^n y(.,t)\Vert _{L^2} =\Vert A z(., t)\Vert _{L^2} \frac{C'}{t^\frac{3}{2}} \Vert A^{n-1} y_0\Vert. 
\]
For $y_0\in L^2(-1,0)$ and $t\in (0,1]$, splitting $[0,t]$ into $[0,\frac{t}{n}]\cup [\frac{t}{n}, \frac{2t}{n} ] \cup \cdots \cup [ \frac{n-1}{n} t, t]$, we obtain
\be
\label{E20}
\Vert A^n y(., t)\Vert _{L^2} 
\le \frac{C'}{ ( \frac{t}{n} )^\frac{3}{2}} \Vert A^{n-1} y (\frac{n-1}{n} t)\Vert \le \cdots 
\le \left( \frac{C'}{(\frac{t}{n} )^\frac{3}{2}}  \right) ^n \Vert  y_0 \Vert = \frac{C'^n}{t^\frac{3n}{2} } n^\frac{3n}{2} \Vert y_0\Vert _{L^2}\cdot 
\ee
If $p\in \N$ is given, we pick $n\in\N$  such that $3n-3\le p\le 3n-1$. Then, by Sobolev embedding, we have that 
\begin{eqnarray*}
\Vert \partial _x ^p y(.,t) \Vert _{L^\infty}
&\le& C \Vert y(.,t)\Vert _{H^{p+1}}\\
&\le& C \Vert y(.,t)\Vert _{H^{3n}} \\
&\le&  C \left( \Vert y(.,t) \Vert _{L^2} +\Vert P y(.,t)\Vert _{L^2} + \cdots + \Vert P^n y(.,t)\Vert _{L^2}  \right) \\
&\le&   C \left( 1 + \frac{C'}{t^\frac{3}{2} } + \cdots + \frac{C'^n}{t^\frac{3n}{2}} n^\frac{3n}{2} \right) \Vert y_0\Vert _{L^2} \\
&\le& C \frac{C'^n (n+1)n^\frac{3n}{2}}{t^\frac{3n}{2}} \Vert y_0\Vert _{L^2} \cdot  
\end{eqnarray*}
Since 
\[
(n+1)\frac{ (3n) ^\frac{3n}{2}}{3^\frac{3n}{2}} \le \frac{(1+\frac{p}{3})}{3^\frac{p+1}{2}}  (p+3) ^\frac{p+3}{2} \le C'' p^\frac{3}{4} (\frac{e}{3}) ^\frac{p}{2} 
((p+1)! )^\frac{1}{2}  
\]
we see that there are some constants $C'''>0$ and $R>0$ such that 
\[
\vert \partial _x ^p \, y (x,t) \vert \le \frac{C'''}{R^p t^\frac{p+3}{2}} (p! )^\frac{1}{2}, \quad 
p\in \N, \ t\in (0,1], \ x\in [0,1].   
\]
From \eqref{E20}, we have that $y\in C((0,1],D(A^n))$ for all $n\ge 0$ and hence that 
$y\in C^ \infty ( [0,1] \times (0,1])$. Finally, for all $n\ge 0$ and $p\ge 0$, we have that 
\begin{eqnarray*}
\partial _t ^n \partial _x ^p \, y &=& (-1)^n P^n \partial _x ^p \, y \\
&=& (-1)^n (\partial _x ^3 + a\partial _x )^n \partial _x ^p \, y \\
&=& (-1) ^n \sum_{q=0}^n \left( \begin{array}{c} n \\ q   \end{array}  \right) 
a^{n-q} \partial _x ^{n+2q +p} \, y   
\end{eqnarray*}
and hence, assuming $R'<1$, 
\begin{eqnarray*}
\vert \partial _t ^n \partial _x ^p \, y(x,t) \vert 
&\le & C'' \sum_{q=0}^n \left( \begin{array}{c} n \\ q   \end{array}  \right) 
a^{n-q} \frac{ (n+2q+p) ! ^\frac{1}{2}}{ R'^{n+2q+p} t ^\frac{n+2q+p+3}{2}} \\
&\le&C'' \frac{(n+1) (2a)^n (3n+p)! ^\frac{1}{2}}{ R'^{3n+p} \, t^\frac{3n+p+3}{2}}  \\
&\le& C''  \frac{(n+1) (2a)^n    2^\frac{3n+p}{2} (3n !) ^\frac{1}{2} p! ^\frac{1}{2}}{ R'^{3n+p}\,  t^\frac{3n+p+3}{2}}  \\
&\le& \frac{K}{ t^\frac{3n+p+3}{2}} \frac{n! ^\frac{3}{2} }{ R_1^n } \frac{p! ^\frac{1}{2} }{ R_2^p}
\end{eqnarray*}
for some $K,C_1,C_2\in (0,+\infty )$ and for all $t\in (0,1]$ and all $x\in [0,1]$. 
The proof of Proposition \ref{prop2} is complete. 
\end{proof}
 It is actually expected that for $y_0\in L^2(-1,0)$ and $u\equiv 0$, we have that 
 \[
 y\in G^{\frac{1}{3}, 1} ([-1,0]\times [\varepsilon , T]) \qquad \forall\  0<\varepsilon <T<\infty. 
 \] 
Proving such a property seems to be challenging. The smoothing effect from $L^2$ to $G^{1/3}$ is much easy to establish on $\R$ for data with compact support. The proof of the following result is given in appendix. 
\begin{prop}
\label{prop20}
Let $y_0\in L^2(\R)$ be such that $y_0(x)=0$ for a.e. $x\in \R\setminus  [ -L, L]$ for some $L>0$. 
Let $y=y(x,t)$ denote the solution of the Cauchy problem 
\ba
\partial _t y + \partial _x ^3 y &=& 0, \quad t>0,\  x\in \R , \label{AN1} \\
y(x,0) &=& y_0(x), \quad x\in \R. \label{AN2}  
\ea
Then $y\in G^{\frac{1}{3}, 1} ([-l,l]\times [ \varepsilon, T])$  for all $l>0$ and all $0<\varepsilon<T$.  
\end{prop}

\subsection{Proof of Theorem \ref{thm1}}
Pick any $y_0\in L^2(-1,0)$, $T>0$, and $s\in [\frac{3}{2}, 3)$. Let $\bar y$ denote the solution of the
free evolution for the KdV system:
\ba
\partial _t \bar y + \partial _x ^3 \bar y + a\partial _x \bar y =0,&&  x\in (-1,0), \ t\in (0,T), \label{K1}\\
\bar y(0,t)=\partial _x \bar y (0,t)= \bar y(-1, t)=0, && t\in (0,T), \label{K2}\\
\bar y(x,0)=y_0(x), && x\in (-1,0). \label{K3}  
\ea
It follows from Proposition \ref{prop2} that  $\bar y\in G^{\frac{1}{2}, \frac{3}{2}} ( [-1,0] \times [\varepsilon, T])$ for any $\varepsilon \in (0,T)$. In particular, $\partial _x ^2 \bar y (0,.) \in G^\frac{3}{2} ([\varepsilon, T])$ for any $\varepsilon \in (0,T)$. Pick any $\tau\in (0,T)$ and let 
\[
z(t)=\phi _s \left( \frac{t-\tau }{T-\tau} \right) \partial _x ^2 \bar y (0,t),  
\]
where $\phi _s$ is the ``step function''
\[
\phi _s (\rho ) = 
\left\{ 
\begin{array}{ll}
1 & \textrm{ if } \rho \le 0, \\
0 & \textrm{ if } \rho \ge 1, \\
\frac{ e^{-\frac{M}{(1-\rho )^\sigma } } }{ e^{-\frac{M}{\rho ^\sigma }} + e^{-\frac{M}{(1-\rho )^\sigma }}      }
&\textrm{ if } \rho \in (0,1), 
\end{array}
\right. 
\]  
with $M>0$ and $\sigma := (s-1)^{-1}$. As $\phi _s$ is Gevrey of order $s$  (see e.g. \cite{MRRschrodinger}) and $s\ge 3/2$, we infer that $z\in G^s( [\varepsilon ,T] )$ for all 
$\varepsilon \in (0,T)$. Let 
\[
y(x,t) = 
\left\{ 
\begin{array}{ll}
y_0(x) & \textrm{ if } x\in [-1,0],\ t=0,\\
\sum_{i\ge 0} g_i (x) z^{(i)} (t) & \textrm{ if } x\in [-1,0], \ t\in (0,T].    
\end{array}
\right. 
\] 
Then, by Proposition \ref{prop1}, $y\in G^{\frac{s}{3}, s} ([-1,0]\times [\varepsilon, T])$ for all $\varepsilon \in (0,T)$, and it 
satisfies   \eqref{B1}-\eqref{B3}. Furthermore, 
\[
\partial _x ^p \, y (0,t)=\partial _x ^p \, \bar y (0,t), \quad \forall t\in (0,T), \ \forall p\in \{ 0,1,2\},  
\]
so that
 \[
y(x,t) = \bar y(x,t) \quad \forall  (x,t)\in [-1,0]\times (0,\tau),
 \]
by Holmgren theorem. We infer that $y\in C([0,T], L^2(-1,0))$ and that it solves  \eqref{A1}-\eqref{A4} if
we define $u$ as in \eqref{B5}. Note that $u(t)=0$ for $0<t<\tau$ and that $u\in G^s([0,T])$. Finally 
$y(.,T)=0$, for $z^{(i)}(T)=0$ for all $i\ge 0$. The proof of Theorem \ref{thm1} is complete. \qed  
\section{Reachable states} 
\label{section3}
\subsection{The limit case $s=3$ in the flatness property}
The following result extends the flatness property depicted in Proposition \ref{prop1} to 
the limit case $s=3$. 
\begin{prop}
\label{prop3}
Assume that $z\in G^3([0,T])$ with 
\be
|z^{ (j) } (t) | \le M \frac{ (3 j) !  }{ R^{3j} } \quad \forall j\ge 0, \ \forall t\in [0,T]
\ee
where $R>1$, and let $y=y(x,t)$ be as in \eqref{B4}. 
Then $y\in G^{1,3} ([-1,0]\times [0,T])$ and it solves  \eqref{B1}-\eqref{B3}. 
\end{prop}
\begin{proof}
We follow closely \cite{MRRreachable}. 
Pick any $m,n\in \N$. By \eqref{Q1}, we can assume that $i\ge n$. Setting $j=3i-3n$ and $N=3n+3m$, so that $j+N=3i+3m$, we have that 
\[
\left\vert 
\partial _t ^m P^n \big(  g_i(x) z^{ (i) } (t) \big)
\right\vert 
\le M \frac{(3i+3m)!}{R^{3i+3m}} \frac{1}{(3(i-n)+2)!} \le M \frac{(j+N)!}{R^{j+N}} \frac{1}{(j+2)!} \cdot
\]
Let 
\[
S:=\sum_{i\ge n} |\partial _t ^m P^n (g_i(x)z^{(i)} (t)) | .
\]
If $N\le 2$, then $S\le M\sum_{j\ge 0} R^{-(j+N)} <\infty$ for $R>1$. Assume from now on that $N\ge 2$. Then 
\begin{eqnarray*}
S 
&\le& M\, \sum_{j\ge 0} \frac{(j+3)\cdots (j+N)}{R^{j+N}}\\
&\le& M\,  \sum_{k\ge 0} \ \ \sum_{kN\le j < (k+1)N} \frac{(j+3)\cdots (j+N)}{R^{j+N}}\\ 
&\le& M\,  \sum_{k\ge 0} N \frac{\big( (k+2)N\big) ^{N-2}}{R^{(k+1)N}} \\
&\le& MN^{N-1} \sum_{k\ge 0} \left( \frac{k+2}{R^{k+1}}\right) ^N\cdot 
\end{eqnarray*} 
Pick any $\sigma\in (0,1)$ and let $a:=\sup_{k\ge 0} \frac{k+2}{(R^{1-\sigma})^{k+1}} <\infty$. We infer from \cite[Proof of Proposition 3.1]{MRRreachable}
that 
\[
\sum_{k\ge 0} \left( \frac{k+2}{R^{k+1}}\right) ^N \le \frac{a^N}{R^{N\sigma}-1} ,
\]
so that 
\[
S\le MN^{N-1} \frac{a^N}{R^{N\sigma}-1} \le M' \left( \frac{ae}{R^\sigma }\right) ^N \frac{N!}{N^\frac{3}{2}}
\]
for some constant $M'>0$, by using Stirling formula. Next, we have that 
\[
N!= (3n+3m)!\le 2^{3n+3m} (3n)! (3m)!
\] 
and using again the estimate
\[
\frac{(3m)!}{m!^3} \sim 3^{3m} \frac{\sqrt{3}}{2\pi m} ,
\]
we infer  that 
\begin{eqnarray*}
S
&\le& M'' \left( \frac{ae}{R^\sigma} \right) ^{3m+3n} 2^{3n+3m} (3n)! \left(  \frac{3^{3m}}{m+1}(m!)^3\right)  \frac{1}{ (m+n+1)^\frac{3}{2} } \\
&\le& M''' \frac{ (3n)! }{R_2^{3n}} \frac{ (m!)^3}{R_1^m} \frac{1}{ (n+1)^\frac{3}{2}}   
\end{eqnarray*} 
for some positive constants $M'',M''',R_1$  and $R_2$. There is no loss of generality in assuming that $R_2<1$. 
Let $K$ be as in Lemma \ref{lem10} for $p=\infty$. Then we have
\begin{eqnarray*}
\sum_{i\ge n} \Vert \partial _t ^m ( g_i(x) z^{( i) } (t) )\Vert _{3n,\infty} 
&\le & K^n \sum_{i\ge n} \ \sum_{0\le j\le n} \Vert \partial _t ^m P^j (g_i(x)z^{(i)} (t))\Vert _\infty \\
&\le& M''' K^n \sum_{0\le j\le n} \left( \frac{(m!)^3}{R_1^m} \frac{(3j)!}{R_2^{3j}} \frac{1}{(j+1)^\frac{3}{2}} \right) \\
&\le& M''' K^n  \frac{(3n)!}{R_2^{3n}} \frac{(m!)^3}{R_1^m} \sum_{j\ge 0} \frac{1}{(j+1)^\frac{3}{2}} \cdot 
\end{eqnarray*}
This shows that the series of derivatives $\partial _t ^m\partial _x ^l\big( g_i (x)z^{ (i) } (t) \big)$ is uniformly convergent 
on $[-1,0]\times [0,T]$ for all $m,l\in \N$, so that $y\in C^\infty ([-1,0]\times [0,T])$, and that the function $y$  satisfies for $l\le 3n$ 
\[
|\partial _t ^m \partial _x ^l \, y (x,t)  |\le M''''   K^n  \frac{(3n)!}{R_2^{3n}} \frac{(m!)^3}{R_1^m} \quad \forall x\in [-1,0], \ \forall t\in [0,T]. 
\] 
for some constant $M''''>0$. Finally, if $l\in \{ 3n-2,3n-1,3n\}$, then $(3n)! (K/R_2^3)^n \le C'l!/R_2'^l$ for some $C'>0$, $R_2'>0$. This yields 
\[
\vert \partial _t^m\partial _x ^l \, y(x,t) \vert \le C'M'''' \frac{(m!)^3}{R_1^m} \frac{l!}{R_2'^l}, \quad \forall x\in [-1,0], \  \forall t\in [0,T]. 
\]
\end{proof}
\subsection{Proof of Theorem \ref{thm2}} 
Pick any $R>R_0=e^{(3e)^{-1}}(1+a)^\frac{1}{3}$ and any $y_1\in {\mathcal R}_R$. Our first task is to write $y_1$ in the form 
\be
\label{R1}
y_1(x)=\sum_{i\ge 0} b_i g_i(x), \quad x\in [-1,0]. 
\ee
Note that if \eqref{R1} holds with a convergence in $W^{n,\infty}(-1,0)$ for all $n\ge 0$, then
\[
\partial _x ^2 P^n y_1 (0)=\sum_{i\ge 0} b_i \partial _x ^2 P^n g_i(0) = \sum_{i\ge n} b_i \partial _x ^2 g_{i-n}(0) =b_n. 
\]  
Set 
\[
b_n:=\partial _x ^2 P^n y_1(0), \quad \forall n\ge 0. 
\]
Since $y_1\in {\mathcal R}_R$ with $R>R_0$, there exists for any $r\in (R_0,R)$ a constant $C=C(r)>0$ such that 
\[
\vert \partial _x ^n \, y_1(x)\vert \le C\frac{n!}{r^n}, \quad \forall x\in [-r,0], \ \forall n\in \N. 
\]
Using Lemma \ref{lem10}, we infer that 
\[
\vert b_n\vert =\vert P^n \partial _x ^2 y_1 (0) \vert \le (1+\frac{1}{a}) (1+a)^n \Vert \partial _x^2 y_1\Vert _{3n,\infty} 
\le C' (1+a)^n \frac{(3n+2)!}{r^{3n+2}}  
\]
for some $C'>0$ and all $n\ge 0$.
We need the following version of Borel Theorem, which is a particular case of \cite[Proposition 3.6]{MRRreachable} 
(with $a_p=[3p(3p-1)(3p-2)]^{-1}$ for $p\ge 1$).
\begin{prop}
\label{prop4}
Let $(d_q)_{q\ge 0}$ be a sequence of real numbers such that 
\[
|d_q|\le C H^q (3q)! \quad \forall q\ge 0
\]
for some $H>0$ and $C>0$. Then for all $\tilde H> e^{e^{-1}} H$, there exists a function $f\in C^\infty (\R )$ 
such that 
\begin{eqnarray*}
f^{(q)} (0)&=& d_q \quad \forall q\ge 0, \\
|f^{(q)} (x) | &\le& C \tilde H ^q (3q)! \quad \forall q\ge 0, \ \forall x\in \R .
\end{eqnarray*}
\end{prop}

Since $r>R_0$, we can pick two numbers $H\in (0,e^{-e^{-1}})$ and $C''>0$ such that 
\[
|b_n|\le C'' H^n (3n)! \quad \forall n\ge 0. 
\]
By Proposition \ref{prop4}, there exists a function $f\in G^3([0,T])$  and a number $R>1$ such that 
\ba
f^{(i)} (T) &=& b_i \quad \forall i\ge 0, \label{V1}\\
\vert f^{( i ) } (t) \vert &\le& C'' \frac{(3i)!}{R^{3i}}\quad  \forall i\ge 0, \ \forall t\in [0,T].  \label{V2}
\ea
Pick any $\tau \in (0,T)$ and let 
\[
g(t)=1 -\phi _2 \left( \frac{t- \tau }{T- \tau}\right) \quad \textrm{ for } t\in [0,T].  
\]
Note that $g\in G^2([0,T])$ and that $g(T)=1$, $g^{( i) }(T)=0$ for all $i\ge 1$. Setting 
\[
z(t)=g(t)f(t) \quad \forall t\in [0,T], 
\] 
we have that $z\in G^3([0,T])$ and that 
\ba
z^{( i) } (T) &=& b_i \quad \forall i\ge 0, \label{M1}\\
z^{( i) } (0) &=& 0 \quad \forall i\ge 0, \label{M2}\\
| z^{(i)} (t)| &\le& C''' \frac{(3i)!}{R^{3i}} \quad \forall i\ge 0, \ \forall t\in [0,T] \label{M3}
\ea
for some $C'''>0$. (The fact that the constant $R>0$ in \eqref{M3} is the same as in \eqref{V2}  is proved as in 
\cite[Lemma 3.7]{MRRreachable}.)
Let $y$ be as in \eqref{B4}. Then by Proposition \ref{prop3} we know that $y\in G^{1,3} ([-1,0]\times [0,T])$ and that it solves
\eqref{B1}-\eqref{B3}.   Let $u(t)=y(-1,t)$ for $t\in [0,T]$. Then $u\in G^3([0,T])$ and $y$ solves \eqref{A1}-\eqref{A4} 
with $y_0=0$, by \eqref{M2}. Furthermore, we have by \eqref{M1} that 
\[
y(x,T)=\sum_{i\ge 0} g_i(x) z^{(i)} (T) = \sum_{i\ge 0} b_i g_i(x), \quad x\in [-1,0].  
\] 
From the proof of Proposition \ref{prop3}, we know that for all $l,m\in \N$, the sequence of partial sums of the series 
$\sum_{i\ge 0} \partial _t^m \partial _x ^l \big( g_i(x) z^{(i)}(t) \big)$ converges uniformly on $[-1,0]\times [0,T]$ to 
$\partial _t^m \partial _x ^l \,  y$, and hence for all $n\ge 0$
\begin{eqnarray*}
&&P^n y(0,T)= \sum_{i\ge 0} b_i P^n g_i (0) = \sum_{i\ge n} b_i g_{i-n }(0)=0, \\
&&\partial _x P^n y(0,T)= \sum_{i\ge 0} b_i \partial _x P^n g_i (0) = \sum_{i\ge n} b_i \partial _x g_{i-n }(0)=0, \\
&&\partial _x ^2 P^n y(0,T)= \sum_{i\ge 0} b_i \partial _x ^2 P^n g_i (0) = \sum_{i\ge n} b_i \partial _x ^2 g_{i-n }(0)=b_n=\partial _x ^2 P^n y_1(0). 
\end{eqnarray*}
To conclude that 
\[
y(x,T)=y_1(x), \quad \forall x\in [-1,0],
\]
it is sufficient to prove the following\\[3mm] 
{\sc Claim 3.} If $h\in G^1([-1,0])$ is such that $P^n h(0)=\partial _x P^n h( 0)=\partial _x ^2 P^n h(0)=0$ for all $n\in \N$, then 
$h\equiv 0$. \\[3mm] 
Indeed, we notice that $P^0=id$ and that $P^n=\partial _x ^{3n} + \cdots$, 
$\partial _x P^n=\partial _x^{3n+1} + \cdots$, and 
$\partial _x^2P^n = \partial _x ^{3n+2} + \cdots$, where $\cdots$ stands for less order derivatives. Then we obtain by induction
that $\partial _x^{3n}h(0)=\partial _x^{3n+1}h(0) =\partial _x^{3n+2}h(0)=0$ for all $n\ge 0$, so that $h\equiv 0$. 
This completes the proof of Claim 3 and of Theorem \ref{thm2}. \qed
\section*{Appendix}
\subsection{Proof of Lemma \ref{lem10}}
We first need to prove two simple lemmas. We still use the notation $P=\partial _x ^3 + a\partial _x$. 
\begin{lem}
\label{lem21}
Let $a\in \R _+$ and $p\in [1,\infty ] $. Then for all $n\in \N$, we have 
\be
\label{P3}
\Vert P^n f\Vert _p\le (1+a)^n \Vert f\Vert _{3n,p}\quad \forall f\in W^{3n,p} (-1,0).
\ee
\end{lem}
\begin{proof}
The proof is by induction on $n$. For $n=0$, the result is obvious. If it is true at rank $n-1$, then 
\[
\Vert P^{n-1} (Pf) \Vert _p \le (1+a)^{n-1} \Vert Pf \Vert _{3n-3,p}  
\le (1+a)^{n-1} \left( \Vert f'''\Vert _{3n-3,p} +  a \Vert f'\Vert _{3n-3,p} \right) \le (1+a) ^n \Vert f\Vert _{3n,p}.    
\] 
\end{proof}
\begin{lem}
\label{lem22}
Let $a\in \R _+$ and $p\in [1,\infty ]$. Then there exists a constant $C_1=C_1(p,a)>0$ such that 
\be
\label{PP0}
\Vert f\Vert _{3,p} \le C_1 (\Vert f\Vert _p + \Vert P\, f\Vert _p)\quad \forall f\in W^{3,p}(-1,0). 
\ee 
\end{lem}
\begin{proof}
The seminorm 
\[
||| f ||| :=\Vert f\Vert _p + \Vert P f \Vert _p 
\]
is clearly a norm in $W^{3,p}(-1,0)$. Let us check that $W^{3,p}(-1,0)$, endowed with the norm $||| \cdot |||$, is a Banach space. Pick any Cauchy sequence 
$(f_n)_{n\ge 0}$ for $|||\cdot |||$ . Then 
\[
||| f_m - f_n ||| =\Vert f_m-f_n\Vert _p + \Vert P f_m-P f_n\Vert _p \to 0, \quad \textrm{ as } m,n\to +\infty.  
\]
Since $L^p(-1,0)$ is a Banach space, there exist $f,g\in L^p(-1,0)$ such that 
\be
\label{P0}
\Vert f_n-f\Vert _p + \Vert P f_n - g \Vert _p\to 0, \quad \textrm{ as } n\to \infty. 
\ee
Since $f_n\to f$ in ${\mathcal D'} (-1,0)$, $Pf_n\to Pf $ in ${\mathcal D'}(-1,0)$ as well, and $Pf=g$. Thus $f''=\int g(x)dx - af\in L^p(-1,0)$, and 
$f\in W^{2,p}(-1,0)$. This yields $f'\in W^{1,p}(-1,0)$ and $f'''=g- a f'\in L^p(-1,0)$, and hence $f\in W^{3,p}(-1,0)$. Note that \eqref{P0} can be written 
$||| f_n-f|||\to 0$. This proves that $(W^{3,p}(-1,0) , ||| \cdot |||)$ is a Banach space. 
Now, applying the Banach theorem to the identity map from the Banach space $(W^{3,p}(-1,0), \Vert \cdot \Vert _{3,p})$ to the Banach space 
 $(W^{3,p}(-1,0), ||| \cdot |||)$, which is linear, continuous, and bijective, we infer that its inverse is continuous; that is,  \eqref{PP0} holds.    
\end{proof}
Let us prove Lemma \ref{lem10}. We proceed by induction on $n$. Both inequalities in \eqref{P1} are obvious for $n=0$. Assume now that both
inequalities in \eqref{P1} are satisfied up to the rank $n-1$. Let us first prove the left inequality in \eqref{P1} at the rank $n$. 
Pick any $f\in W^{3n,p} ( -1,0)$. Then, by the induction hypothesis and Lemma  \ref{lem21},  we obtain
\begin{eqnarray*}
\sum_{i=0}^n \Vert P^i f\Vert _p &=& \sum_{i=0}^{n-1} \Vert P^i f\Vert _p + \Vert P^n f\Vert _p \\
&\le& (1+\frac{1}{a}) (1+a)^{n-1} \Vert f\Vert _{3n-3,p} +(1+a)^n \Vert f\Vert _{3n,p} \\
&\le& (1+\frac{1}{a}) (1+a)^n \Vert f\Vert _{3n,p},
\end{eqnarray*}
as desired. 

For the right inequality in \eqref{P1}, we write 
\ba
\Vert f\Vert _{3n,p} &=& \Vert f\Vert _{3n-3,p} 
+ \Vert \partial _x ^{3n-2} f\Vert _p +  \Vert \partial _x ^{3n-1} f\Vert _p +  \Vert \partial _x ^{3n} f\Vert _p \nonumber \\
&\le& K^{n-1} \sum_{i=0}^{n-1} \Vert P^i f\Vert _p + \Vert \partial _x ^{3n-3} f \Vert _{3,p}. \label{P4}  
\ea
But it follows from Lemma \ref{lem22} that 
\begin{eqnarray*}
\Vert \partial _x ^{3n-3} f \Vert _{3,p} 
&\le &  C_1 (\Vert \partial _x ^{3n-3} f\Vert _p + \Vert P \partial _x ^{3n-3} f\Vert _p ) \\
&\le &  C_1 (\Vert \partial _x ^{3n-3} f\Vert _p + \Vert \partial _x ^{3n-3} P \,  f\Vert _p ) \\
&\le &  C_1 (\Vert f\Vert _{3n-3,p} + \Vert P f \Vert _{3n-3,p} ) \\
&\le &  C_1 K^{n-1} \left( \sum _{i=0}^{n-1} \Vert P^i f\Vert _p + \sum_{i=1}^n \Vert P^i f\Vert _p\right) \\
&\le& 2C_1K^{n-1}      \sum _{i=0}^{n} \Vert P^i f\Vert _p. 
\end{eqnarray*}
Combined with \eqref{P4}, this yields
\[
\Vert f\Vert _{3n,p} \le (1+ 2C_1)K^{n-1}  \sum _{i=0}^{n} \Vert P^i f\Vert _p.
\]
It is sufficient to pick $K:=1+2C_1$.\qed

\subsection{Proof of Proposition \ref{prop20}}
Let $\textrm{Ai}$ denote the Airy function defined as the inverse Fourier transform of $\xi\to 
\exp(i\xi ^3/3)$. Then it is well known (see e.g. \cite{hormander})   that $\textrm{Ai}$ is an entire 
(i.e. complex analytic on $\C$) function  satisfying 
\be
\label{AN3}
\textrm{Ai} ''(x)=x\textrm{Ai} (x), \quad \forall x\in \C. 
\ee 
To prove that $y\in G^{\frac{1}{3}, 1} ( [-l,l] \times [\varepsilon , T])$, we need first check that the Airy function is itself Gevrey of order $1/3$. \\
{\sc Claim 4.}  $\textrm{Ai}\in G^\frac{1}{3} ([-l,l])$ for any $l>0$.  \\
Let us prove Claim 4. Differentiating $k+1$ times in \eqref{AN3} and letting $x=0$ results in 
\[
\textrm{Ai}^{(3+k)} (0)=(\textrm{Ai}'')^{(k+1)} (0)=(k+1) \textrm{Ai} ^{(k)} (0)\quad \forall k\in \N.  
\]
Note that $\textrm{Ai} ''(0)=0$ by \eqref{AN3} and that 
\[
\textrm{Ai} (0)=[3^\frac{2}{3} \Gamma (\frac{2}{3} )  ]^{-1}\ne 0,  \quad 
\textrm{Ai}'(0)= - [3^\frac{1}{3} \Gamma (\frac{1}{3}) ]^{-1} \ne 0 .\]  
We infer that for all $k\in \N$
\[
\textrm{Ai}^{(3k)} (0)=\textrm{Ai} (0) \prod _{j=1}^{k-1} (1+3j), \quad 
\textrm{Ai}^{(3k+1)} (0)=\textrm{Ai}' (0) \prod _{j=1}^{k-1} (2+3j),\quad
\textrm{Ai}^{(3k+2)} (0)=0.
\]
As it was noticed in \cite[Remark 2.7]{MRRparabolic},  we have 
\[
s^i \, i! \le \prod _{j=1}^i (r+js) \le s^i (i+1)!\quad \forall s\in \N, \ \forall r\in [0,s] .  
\]
Thus there is some constant $C_1>0$ such that 
\[
\left\vert \textrm{Ai} ^{(3k+q)} (0) \right\vert 
\le C_1 3^k k!\qquad \forall k\ge 0,  \ \forall q\in \{ 0,1,2\} . 
\]
From Stirling formula we have $(3k)!\sim \frac{\sqrt{3}}{2\pi k} (3^k k!)^3$ so that 
\be
\label{AN6}
\vert \textrm{Ai} ^{(n)} (0)\vert \le C_2 [(n+1)!]^\frac{1}{3} \le C_3 \frac{n!^\frac{1}{3}}{R^n} \qquad 
\forall n\in \N   
\ee
for any $R\in (0,1)$ and some constants $C_2,C_3>0$. The following result comes from \cite{MRRschro,MRRevolution}.
\begin{lem}
\label{lem30}
Let $s\in (0,1)$ and let $(a_n)_{n\ge 0}$ be a sequence such that 
\[
\vert a_n\vert  \le C \frac{n!^s}{R^n} \quad \forall n\ge 0
\]
for some constants $C,R>0$. Then the function $f(x)=\sum_{n\ge 0} a_n\frac{x^n}{n!}$ is Gevrey of order $s$ on $[-l,l]$ for all $l>0$, with 
\[
\vert f^{(m)} (x) \vert \le C\left( \sum_{k\ge 0 } \frac{(2^sl)^k}{R^k k! ^{1-s}} \right)  \frac{2^{sm}}{R^m} m!^s \quad \forall m\in \N , \ \forall x\in [-l,l]. 
\]
\end{lem}
It follows from \eqref{AN6} and Lemma \ref{lem30} that  the Airy function is Gevrey of order $1/3$ on each 
interval $[-l,l]$ with 
\[ 
\vert \textrm{Ai} ^{(m)} (x) \vert  \le C_4(R,l) \frac{2^\frac{m}{3}}{R^m} m! ^\frac{1}{3} \qquad \forall m\in \N, \ \forall x\in [-l,l]. 
\]

Claim 4 is proved. Let us go back to the proof of Proposition \ref{prop20}.   
The fundamental solution of the simplified linear Korteweg-de Vries equation \eqref{AN1} is given by 
\be
E (x,t)= \frac{1}{ (3t)^\frac{1}{3} } \textrm{Ai} \left(  \frac{x}{ (3t)^\frac{1}{3} } \right), 
\ee 
so that the solution of \eqref{AN1}-\eqref{AN2} for an initial value $y_0\in L^2(\R )$ supported in $[-L,L]$ reads
\[
y(x,t)=[E(.,t)*y_0](x)=  \frac{1}{ (3t)^\frac{1}{3} } \int_{-L}^L \textrm{Ai} 
\left(   \frac{x-s}{ (3t)^\frac{1}{3} } \right)  y_0(s)\,  ds. 
\] 
It is clear that $y$ is of class $C^\infty$ on $\R _x \times (0,+\infty )_t$. Pick any $l>0$ and any 
$0<\varepsilon <T$.  Since $y$ solves \eqref{AN1}, we have for all $x\in [-l,l]$, $t\in [\varepsilon , T]$
and $p,q\in \N$ that 
\begin{eqnarray*}
\vert \partial _x^p\partial _t^q y(x,t) \vert 
&=& \vert \partial _x ^{p+3q} y(x,t)\vert \\
&\le& K \frac{2^\frac{p+3q}{3} }{(3t)^{\frac{1+p+3q} {3}}  R^{p+3q} } (p+3q) ! ^\frac{1}{3} \Vert y_0\Vert _{L^1(-L,L)} \\
&\le& K \frac{2^\frac{p+3q}{3} }{ (3\varepsilon)^{\frac{1+p+3q} {3}}   R^{p+3q} } \big( 2^{p+3q} p! (3q)!  \big)^\frac{1}{3} \Vert y_0\Vert _{L^1(-L,L)}\\
&\le& \frac{K'}{R_1^pR_2^q } p !^\frac{1}{3} q!  \Vert y_0\Vert _{L^1(-L,L)} \\
\end{eqnarray*}
for some constants $K,K',R_1,R_2>0$ which depend on $l$, $L$, $\varepsilon$ and $T$. 
The proof of Proposition \ref{prop20} is complete. \qed
\section*{Acknowledgements} This work was done when the second author (IR)  was visiting CAS, MINES ParisTech. The third author (LR) was supported by the ANR project Finite4SoS (ANR-15-CE23-0007).

%%%%%%%%%%%%%%%%%%%%%%%%%%%%%%%%%%%%%%%%%%%%%%%%%%%%%%%%%%%%%%%%%%%%%%%%%%%%%%%%%%%%%%%%%%%%%%%%%%%%%%%%%%%%%%%%%%%%%%%%%%%%%%%

\end{document}